\documentclass[english,12pt]{amsart}
\usepackage{amsmath,amsfonts,bbm,amsthm}
\usepackage[latin1]{inputenc}  % accents 8 bits dans le source
\usepackage[T1]{fontenc}       % accents dans le DVI
\usepackage[psamsfonts]{amssymb}

\usepackage{latexsym}
\usepackage[dvips]{epsfig}
\usepackage{dcolumn}
\usepackage{tabularx}
\usepackage{longtable}
\usepackage{graphics}
\usepackage{pstricks}
\usepackage{pst-node}
\usepackage{latexsym}
\usepackage{babel}
\usepackage{euscript}

\renewcommand{\theequation}{\thesection.\arabic{equation}}
\newtheorem{theorem}{Theorem}
\newtheorem{lemma}{Lemma}
\newtheorem{proposition}{Proposition}
\newtheorem{corollary}{Corollary}
\newtheorem{remark}{Remark}

\newcommand{\eqnsection}{
\renewcommand{\theequation}{\thesection.\arabic{equation}}
    \makeatletter
    \csname  @addtoreset\endcsname{equation}{section}
    \makeatother}
\eqnsection
\def\e{{\mathbb E}}
\def\demi{{1\over 2}}

\def\Z{{{\Bbb Z}}}
\def\P{{{\Bbb P}}}
\def\N{{{\Bbb N}}}
\def\R{{{\Bbb R}}}

\def\lll{{{\mathcal L}}}

\def\fff{{{\mathcal F}}}

\def\ttt{{\mathcal T}}

\def\ue{{\underline e}}
\def\oe{{\overline e}}
\def\E{{{\Bbb E}}}

\def\hhh{{{\mathcal H}}}

\def\indic{{{\mathbbm 1}}}

\newtheorem{corol}{Corollary}[section]
\newtheorem{rema}{Remark}[section]
\newtheorem{exa}{Example}

\renewcommand{\le}{\leqslant}

\renewcommand{\ge}{\geqslant}

\renewcommand{\subset}{\subseteq}

\newcommand{\bal}{\begin{align*}}
\newcommand{\eal}{\end{align*}}
\newcommand{\beq}{\begin{eqnarray*}}
\newcommand{\eeq}{\end{eqnarray*}}
\newcommand{\bte}{\begin{theorem}}
\newcommand{\ete}{\end{theorem}}
\newcommand{\bl}{\begin{lemma}}
\newcommand{\el}{\end{lemma}}
\newcommand{\bd}{\begin{description}}
\newcommand{\ed}{\end{description}}

\newcommand{\bc}{\begin{cases}}
\newcommand{\ec}{\end{cases}}
\newcommand{\bp}{\begin{proof}}
\newcommand{\ep}{\end{proof}}
\newcommand{\bco}{\begin{corol}}
\newcommand{\eco}{\end{corol}}

\newcommand{\1}{\hbox{1 \hskip -7pt I}}

\newcommand{\iy}{\infty}
\newcommand{\tx}{\text}

\newcommand{\tl}{\tilde{l}}

\newcommand{\te}{\tilde{E}}

\newcommand{\tX}{\tilde{X}}

\newcommand{\F}{\mathcal{F}}

\newcommand{\Es}{\mathbb{E}}
\newcommand{\Pb}{\mathbb{P}}

\newcommand{\T}{\mathcal{T}}

\newcommand{\be}{\beta}

\newcommand{\g}{\gamma}

\newcommand{\G}{\Gamma}

\newcommand{\de}{\delta}

\newcommand{\la}{\lambda}
\newcommand{\La}{\Lambda}

\newcommand{\Om}{\Omega}

\def\bdes{\begin{description}}
\def\edes{\end{description}}

\def\tL{{\tilde L}}

\begin{document}
\author[C. Sabot]{Christophe SABOT}
\address{Universit\'e de Lyon, Universit\'e Lyon 1,
Institut Camille Jordan, CNRS UMR 5208, 43, Boulevard du 11 novembre 1918,
69622 Villeurbanne Cedex, France} \email{sabot@math.univ-lyon1.fr}
\author[P. Tarres]{\vspace{3mm} Pierre Tarres\\
{\it Accepted for publication in the \\Journal of the
European Mathematical Society}}
\address{Mathematical Institute, University of Oxford, Andrew Wiles Building, Radcliffe Obervatory Quarter, Woodstock Road, Oxford OX2 6GG. } \email{tarres@maths.ox.ac.uk}

\title[ ]{Edge-reinforced random walk, Vertex-Reinforced Jump Process and the supersymmetric hyperbolic sigma model}

\keywords{} \subjclass[2000]{primary 60K37, 60K35,
secondary 81T25, 81T60}
\thanks{This work was partly supported by the ANR projects MEMEMO and MEMEMO2, and by a Leverhulme Prize. }

\maketitle
%\footnotetext[1]{Universit\'e de Lyon, Universit\'e Lyon 1,
%Institut Camille Jordan, CNRS UMR 5208, 43, Boulevard du 11 novembre 1918,
%69622 Villeurbanne Cedex, France. E-mail: sabot@math.univ-lyon1.fr}
%\footnotetext[2]{Universit\' e Paul Sabatier, Institut de Math\'ematiques de Toulouse, CNRS UMR 5219, 
%118 route de Narbonne,
%Toulouse cedex 9,
%France. E-mail: pierre.tarres@math.univ-toulouse.fr. On leave from the Mathematical Institute, University of Oxford.}
\begin{abstract}
Edge-reinforced random walk (ERRW), introduced by Coppersmith and Diaconis in 1986 \cite{coppersmith}, is a random process, which takes values in the vertex set of a graph $G$, and is more likely to cross edges it has visited before. We show that it can be represented in terms of a  Vertex-reinforced jump process (VRJP) with independent gamma conductances: the VRJP was conceived by Werner and first studied by Davis and Volkov \cite{dv1,dv2}, and is a continuous-time process favouring sites with more local time. We calculate, for any finite graph $G$, the limiting measure of the centred occupation time measure of VRJP, and interpret it  as a supersymmetric hyperbolic sigma model in quantum field theory \cite{dsz}.  This enables us to deduce that VRJP and ERRW are positive recurrent in any dimension for large reinforcement, and that VRJP is transient in dimension greater than or equal to $3$ for small reinforcement, using results of Disertori and Spencer \cite{ds}, Disertori, Spencer and Zirnbauer \cite{dsz}.
\end{abstract}

\section{Introduction}

Let $(\Om,\F,\Pb)$ be a probability space. Let  $G=(V,E,\sim)$ be a nonoriented connected locally finite graph without loops. Let $(a_{e})_{e\in E}$ be a sequence of positive initial weights associated to each edge $e\in E$.

Let $(X_n)_{n\in\N}$ be a random process that takes values in $V$, and let $\F_n=\sigma(X_0,\ldots,X_n)$ be the filtration of its past. For any $e\in E$, $n\in\N\cup\{\iy\}$, let 
\begin{equation}
Z_n(e)=a_e+ \sum_{k=1}^n\1_{\{\{X_{k-1},X_k\}=e\}}
\end{equation}
be the number of crosses of $e$ up to time $n$ plus the initial weight $a_e$.

Then $(X_n)_{n\in\N}$ is called Edge Reinforced Random Walk (ERRW) with starting point $i_0\in V$ and weights $(a_e)_{e\in E}$,  
if  $X_0=i_0$ and,  for all $n\in\N$, 
\begin{equation}
\label{def-vsirw}
\Pb(X_{n+1}=j~|~\F_n)=\indic_{\{j\sim X_n\}}\frac{Z_n(\{X_n,j\})}
{\sum_{k\sim X_n} Z_n(\{X_n,k\})}.
\end{equation}

The Edge Reinforced Random Walk was introduced in 1986 by Diaconis \cite{coppersmith}; on finite 
graphs it is a mixture of reversible Markov chains, and the mixing measure can be determined explicitly 
(the so-called Coppersmith-Diaconis measure, or "magic formula" \cite{diaconis2}, 
see also \cite{keane-rolles1,rolles1}),  which has applications in Bayesian statistics \cite{diaconis-rolles,bacallado1,bacallado2}.

On infinite graphs, the research has focused so far on recurrence/transience criteria. On acyclic or directed graphs, the walk can be seen as a random walk in an \textit{independent } random environment \cite{pemantle4}, and a recurrence/transience phase transition was first observed by Pemantle on trees \cite{collevecchio1,keane-rolles2,pemantle4}. In the case of infinite graphs with cycles, recurrence criteria and asymptotic estimates were obtained by Merkl and Rolles on graphs of the form $\Z\times G$, $G$ finite graph, and on a certain two-dimensional graph \cite{merkl-rolles1,merkl-rolles4,merkl-rolles2,rolles2}, but recurrence on $\Z^2$ was still unresolved.   

Also, this original ERRW model \cite{coppersmith} has triggered a number of similar models of self-organization and learning behaviour; see for instance Davis \cite{davis}, Limic and Tarr\`es \cite{limic,limic-tarres}, Pemantle \cite{Pem07}, Sabot \cite{sabot1,sabot2}, Tarr\`es \cite{tarres2,tarres3} and T\'oth \cite{toth3}, with different perspectives on the topic.

Our first result relates the ERRW to the Vertex-Reinforced Jump Process (VRJP), conceived by Werner and studied by Davis and Volkov \cite{dv1,dv2}, Collevechio \cite{collevecchio2,collevecchio3} and Basdevant and Singh \cite{bs}. 

We call VRJP with weights $(W_e)_{e\in E}$ a continuous-time process $(Y_t)_{t\ge0}$ on $V$, starting at time $0$ at some vertex $i_0\in V$ and such that, if $Y$ is at a vertex $i\in V$ at time $t$, then, conditionally on $(Y_s, s\le t)$, the process jumps to a neighbour $j$ of $i$ at rate $W_{\{i,j\}}L_j(t)$, where
$$L_j(t):=1+\int_0^t \1_{\{Y_s=j\}}\,ds.$$

The main results of the paper are the following.
In Section 2, Theorem \ref{annealed},  we represent the ERRW in terms of a  VRJP with independent gamma conductances. 
Section 3 is dedicated to showing, in Theorem \ref{meas}, that the VRJP is a mixture of time-changed Markov jump processes, 
with a computation of the mixing law.
In Section 6, we interpret that mixing law with the supersymmetric hyperbolic sigma model introduced by Disertori, Spencer and Zirnbauer in \cite{dsz} and related to the Anderson model. We prove positive recurrence of VRJP and ERRW in any dimension for large reinforcement in Corollaries \ref{rec} and \ref{rec2}, using a  localization result of Disertori and Spencer \cite{ds}, and transience of VRJP in dimension $d\ge3$ at small reinforcement in Corollary \ref{trans} using a delocalization result of Disertori, Spencer, Zirnbauer \cite{dsz}. 

Shortly after this paper appeared electronically, Angel, Crawford and Kozma \cite{ack} proposed another proof of recurrence of ERRW under similar assumptions, without making the link with statistical physics. Equivalent results were also obtained for the VRJP, using the existence of a limiting environment proved in this paper. 
\section{From ERRW to VRJP.}
It is convenient here to consider a time changed version of $(Y_s)_{s\ge0}$: consider the positive continuous additive functional of $(Y_s)_{s\ge0}$
$$
A(s)= \int_0^s {1\over L_{Y_u}(u)} du = \sum_{x\in V} \log(L_x(s)) ,
$$ 
and the time changed process
$$
X_t= Y_{A^{-1}(t)}.
$$
Let $(T_i(t))_{i\in V}$ be the local time of the process $(X_t)_{t\ge0}$
$$
T_x(t)=\int_0^t \indic_{\{X_u=x\}} du.
$$
\begin{lemma}
\label{invtime}
The inverse functional $A^{-1}$ is given by
$$
A^{-1}(t)= \int_0^t e^{T_{X_u}(u)} du=\sum_{i\in V}(e^{T_i(t)}-1).
$$
The law of the process $X_t$ is described by the following: conditioned on the past at time $t$, if the process $X_t$ is at the position $i$, then
it jumps to a neighbor $j$ of $i$ at rate
$$
W_{i,j} e^{T_i(t)+T_j(t)}.
$$
\end{lemma}
\begin{proof}
First note that
\begin{equation}
\label{tlx}
T_x(A(s))= \log(L_x(s)),
\end{equation}
since
$$
(T_x(A(s)))'= A'(s)\indic_{\{X_{A(s)}=x\}}={1\over L_{Y_s}(s)} \indic_{\{Y_s=x\}}.
$$
Hence,
\begin{eqnarray*}
(A^{-1}(t))'&=&{1\over A'(A^{-1}(t))}
=
L_{X_t} (A^{-1}(t))
=
e^{T_{X_t}(t)},
\end{eqnarray*}
which yields the expression for $A^{-1}$.
It remains to prove the last assertion:
\begin{eqnarray*}
\P(X_{t+dt}=j | \fff_t) &=&
\P(Y_{A^{-1}(t+dt)}=j | \fff_t)
\\
&=&
W_{X_t,j}(A^{-1})'(t) L_{j}(A^{-1}(t)) dt
\\
&=&
W_{i,j} e^{T_{X_t}(t)}e^{T_j(t)} dt
\end{eqnarray*}
\end{proof}

In order to relate ERRW to VRJP, let us first define the following process $(\tX_t)_{t\in\R_+}$, initially introduced by Rubin, Davis and Sellke \cite{davis,sellke2}, which we call here continuous-time ERRW with weights $(a_e)_{e\in E}$ and starting at  $\tX_0:=i_0$ at time $0$.   

\begin{itemize}
\item
Define on each edge $e\in E$ independent point processes (alarm times) as follows. Let $(\tau^e_k)_{e\in E,k\in\Z_+}$ be independent exponential random variables with parameter 1 and define 
$$
V^e _k =\sum_{l=0}^{k-1}{1\over a_e+l} \tau_l^e,\;\;\; \forall k\in\N.$$
\item
Each edge $e\in E$ has its own clock, denoted by $\tilde T_e(t)$, which only runs when the process $(\tX_t)_{t\ge0}$ is adjacent to $e$. This means that if $e=\{i,j\}$, then $\tilde T_{\{i,j\}}(t)= \tilde T_i(t)+\tilde T_j(t)$,
where $\tilde T_i(t)$ is the local time of the process $\tilde X$ at vertex $i$ and time $t$. 
\item
When the clock of an edge $e\in E$ rings, i.e. when $\tilde T_e(t)= V^e_k$ for some $k>0$, then $\tX_t$ crosses it instantaneously (of course, this can happen only when $\tilde X$ is adjacent to $e$).
\end{itemize}

%%%%%%%%%%%%%%%%%%FIGURE 2%%%%%%%%%%%%%%%%%%%%%%%%%

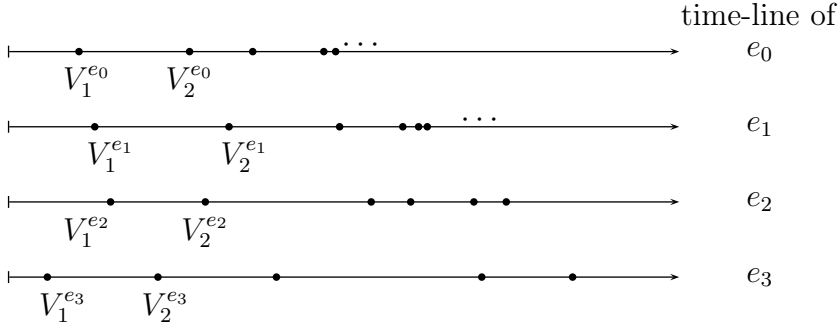
\psset{xunit= 1.05cm,yunit= 1cm}
%\psset{xunit= 0.4725cm,yunit= 0.45cm}
\pspicture(0,-2.5)(15,2)
\psset{linewidth=.5pt}
%draw time lines
\psline{|->}(0,0)(8.5,0)
\psline{|->}(0,1)(8.5,1)
\psline{|->}(0,-1)(8.5,-1)
\psline{|->}(0,-2)(8.5,-2)

\pscircle*(1.1,0) {.05} 
\pscircle*(2.8,0) {.05} 
\pscircle*(4.2,0){.05}
\pscircle*(5,0){.05}
\pscircle*(5.2,0){.05}
\pscircle*(5.31,0){.05}

\pscircle*(0.9,1){.05} 
\pscircle*(2.3,1){.05} 
\pscircle*(3.1,1){.05}
\pscircle*(4,1){.05}
\pscircle*(4.15,1){.05}

\pscircle*(1.3,-1) {.05} 
\pscircle*(2.5,-1) {.05} 
\pscircle*(4.6,-1){.05}
\pscircle*(5.1,-1){.05}
\pscircle*(5.9,-1){.05}
\pscircle*(6.31,-1){.05}

\pscircle*(0.5,-2){.05} 
\pscircle*(1.9,-2){.05} 
\pscircle*(3.4,-2){.05}
\pscircle*(6,-2){.05}
\pscircle*(7.15,-2){.05}

%text
\rput[u](1,0.6){$V_1^{e_0}$}
\rput[u](2.3,0.6){$V_2^{e_0}$}
\rput[u](1.3,-.4){$V_1^{e_1}$}
\rput[u](3,-.4){$V_2^{e_1}$}
\rput[u](1,-1.4){$V_1^{e_2}$}
\rput[u](2.5,-1.4){$V_2^{e_2}$}
\rput[u](0.7,-2.4){$V_1^{e_3}$}
\rput[u](2,-2.4){$V_2^{e_3}$}
%ldots
\rput[b](4.5,1){$\cdots$}
\rput[b](6,0){$\cdots$}
%more text
\rput[u](9.5,1.5){time-line of}
\rput[l](9.35,1){$e_0$}
\rput[l](9.35,0){$e_1$}
\rput[l](9.35,-1){$e_2$}
\rput[l](9.35,-2){$e_3$}
\endpspicture

\medskip

Let $\tau_n$ be the $n$-th jump time of $(\tX_t)_{t\ge0}$, with the convention that $\tau_0:=0$.
\bl
\label{icd} (Davis \cite{davis}, Sellke \cite{sellke2})
Let $(X_n)_{n\in\N}$ (resp. $(\tX_{t})_{t\ge0}$) be an ERRW (resp. continuous-time ERRW) with weights $(a_e)_{e\in E}$, starting at some vertex $i_0\in V$.  
Then  $(\tX_{\tau_n})_{n\ge0}$ and $(X_n)_{n\ge0}$ have the same distribution.
\el
\bp
The argument is based on the memoryless property of exponentials, and on the observation that, if $A$ and $B$ are two independent random variables of parameters $a$ and $b$, then $\Pb[A<B]=a/(a+b)$.
\ep

\begin{theorem}
\label{annealed}
Let $(\tX_{t})_{t\ge0}$ be a continuous-time ERRW with weights $(a_e)_{e\in E}$. 
Then there exists a sequence of independent random variables $W_e\sim\it{Gamma}(a_e,1)$, $e\in E$, such that, conditionally on $(W_e)_{e\in E}$, $(\tilde X_{t})_{t\ge0}$ has the same law as the time modification $(X_t)_{t\ge 0}$ of the VRJP with weights $(W_e)_{e\in E}$. 

In particular, the ERRW $(X_n)_{n\ge 0}$ is equal in  law to the discrete time process associated with a 
VRJP in random independent conductances $W_e\sim \it{Gamma}(a_e,1)$.
\end{theorem}
\bp
For any $e\in E$, define the simple birth process $\{N_t^{e},t\ge0\}$ with initial population size $a_e$, by 
$$N_t^{e}:=a_e+\sup\left\{k\in\N\tx{ s.t. }V_k^{e}\le t\right\}.$$
This process is sometimes called the Yule process: by a result of D. Kendall \cite{kendall}, there exists $W_e:=\lim N_t^{e}e^{-t}$, with distribution $Gamma(a_e,1)$, such that, conditionally on $W_e$, $\{N^{e}_{f_{W_e}(t)}, t\ge0\}$ is a Poisson point process with unit parameter, where $$f_W(t):=\log(1+t/W).$$ 

Let us now condition on $(W_e)_{e\in E}$: $N^{e}$ increases between times $t$ and $t+dt$ with probability 
$W_ee^t\,dt=(f^{-1}_{W_e})'(t)\,dt$. A similar characterization of the timelines is also used in \cite{tarres3}, Lemma 4.7.
If $\tX$ is at vertex $x$ at time $t$, it jumps to a neighbour $y$ of $x$ at rate $W_{x,y}e^{T_x(t)+T_y(t)}$. 
% And $\hX(t)$ is a VRJP with weights $(W_e)_{e\in E}$ since,  for all $e$, $N_{\log(1+t)}^{e}$ is a Poisson process with parameter $W_e$. 
\ep
\section{The mixing measure of VRJP.}

Next we study VRJP. Given fixed weights $(W_e)_{e\in E}$, we denote by $(Y_t)_{t\ge 0}$ the VRJP and  $(X_t)_{t\ge 0}$ its time modification defined in the previous Section, starting at site $X_0:=i_0$ at time $0$ and $(T_i(t))_{i\in V}$ its local time. 
%and such that, if $\tX$ is at vertex $x$ at time $t$, it jumps to a neighbour $y$ of $x$ with probability $W_{x,y}e^{T_x+T_y}\,dt$, where 
%$$
%T_x=T_x(t):= \int_0^{t} \indic_{\{X_u=x\}} du.
%$$
%is the time spent by $(\tX_t)_{t\ge0}$ at site $x$ before time $t$. For notational purposes, we will study $\tX$ here, instead of $\hX$ as is usually the %case. 

It is clear from the definition that the joint process $\Theta_t=(X_t, (T_i(t))_{i\in V})$ is a time continuous Markov process on the state space 
$V\times \R_+^V$  with generator $\tilde L$ defined on $C^\infty$ bounded functions by
$$
\tilde L (f)(i,T)= \left ({\partial \over \partial T_i}f\right)(i,T) + L(T)(f(.,T))(i), \;\;\; \forall (x,T)\in V\times \R_+^V,
$$
where $L(T)$ is the generator of the jump process on $V$ at frozen $T$ defined for $g\in \R^V$:
$$
L(T)(g)(i)=\sum_{j\in V} W_{i,j} e^{T_i+T_j} (g(j)-g(i)), \;\;\; \forall i\in V.
$$
We denote by $\P_{i_0,T}$ the law of the Markov process with generator $\tilde L$ starting from the initial state $(i_0,T)$. 

Note that the law of $(X_t, T(t)-T)$ under $\P_{i_0,T}$ is equal to the law of the process starting from $(i_0, 0)$ with conductances 
$$
W_{i,j}^T=W_{i,j}e^{T_i+T_j}.
$$
For simplicity, we let $\P_i:=\P_{i,0}$.

We show, in Proposition \ref{pconv}, that for finite graphs
the centred occupation times converge a.s., and calculate the limiting measure in Theorem  \ref{meas} i). 
In Theorem \ref{meas} ii) we show that the VRJP $(Y_s)_{s\ge 0}$ (as well as $(X_t)_{t\ge 0}$) is a mixture of time-changed Markov jump processes. 

This limiting measure can be interpreted as a supersymmetric hyperbolic sigma model.
We are grateful to a few specialists of field theory for their advice: Denis Perrot who mentioned that the limit measure of VRJP could be related to the sigma model, and Krzysztof Gawedzki who pointed out reference \cite{dsz}, which actually mentions a possible link of their model with ERRW, suggested by Kozma, Heydenreich and Sznitman, cf \cite{dsz} Section 1.5.  

Note that when $G$ is a tree, if the edges are for instance oriented  towards the root, letting $V_e= e^{U_{\oe}-U_{\ue}}$, the random variables $(V_e)$ are independent and are distributed according to an inverse gaussian law. This was understood in previous works on VRJP \cite{dv1,dv2,collevecchio2,collevecchio3,bs}.

Theorems \ref{annealed} and \ref{meas} enable us to retrieve, in Section  \ref{BCP} the limiting measure of ERRWs, computed by Coppersmith and Diaconis in \cite{coppersmith} (see also \cite{keane-rolles1}), by integration over the random gamma conductances $(W_e)_{e\in E}$. This explains its renormalization constant, which had remained mysterious so far.
\begin{proposition}
\label{pconv}
Suppose that $G$ is finite and set $N=\vert V \vert$.  For all $i\in V$, the following limits exist $\P_{i_0}$ a.s.
$$
U_i =\lim_{t\to \infty} T_i(t) - {t\over N}.
$$
\end{proposition}

\begin{theorem}
\label{meas}
Suppose that $G$ is finite and set $N= \vert V\vert$.

\noindent i) Under $\P_{i_0}$, $(U_i)_{i\in V}$ has the following density distribution on $\hhh_0=\{(u_i), \; \sum u_i=0\}$
\begin{eqnarray}\label{density}
{1\over (2\pi)^{(N-1)/2}} e^{u_{i_0}} e^{-H(W,u)} \sqrt{D(W,u)},
\end{eqnarray}
where
$$
H(W,u)=2 \sum_{\{i,j\}\in E} W_{i,j} \sinh^2\left({\demi(u_i-u_j)}\right)
$$
and $D(W,u)$ is any diagonal minor of the $N\times N$ matrix $M(W,u)$ with coefficients
$$
m_{i,j}=\left\{ \begin{array}{ll} W_{i,j} e^{u_i+u_j}  & \hbox{ if $i\neq j$}
\\
-\sum_{k\in V} W_{i,k} e^{u_i+u_k} &\hbox{ if $i=j$}
\end{array}\right.
$$

\noindent ii)  Let $C$, resp. $D$, be positive continuous additive functionals of $X$, resp. $Y$:
$$C(t)=\sum_{i\in V}(e^{2T_i(t)}-1),\;\;\; D(s)=\sum_{i\in V} L_i^2(s)-1$$
and let $$Z_t=X_{C^{-1}(t)} \; (= Y_{ D^{-1}(t)}).$$ 
 Then, conditionally on $(U_i)_{i\in V}$, $Z_t$  is a Markov jump process starting from  $i_0$, with jump rate from  $i$ to $j$
$$
\demi W_{i,j}e^{U_j-U_i}.
$$
In particular, the discrete time process associated with $(Y_s)_{s\ge0}$ is a mixture of reversible Markov chains with conductances
$W_{i,j} e^{U_i+U_j}$.
\end{theorem}
\noindent N.B.: 1) the density distribution in \eqref{density} is with respect to the Lebesgue measure on $\hhh_0$ which is
$\prod_{i\in V\setminus\{j_0\}} du_i$ for any choice of $j_0$ in $V$. We simple write $du$ for any of the $\prod_{i\in V\setminus\{j_0\}} du_i$.

\noindent 2) The diagonal minors of the matrix $M(W,u)$ are all equal since the sum on any line or column
of the coefficients of the matrix are null. By the matrix-tree theorem, if we let $\T$ be the set of spanning trees of $(V,E,\sim)$, then $D(W,u)=\sum_{T\in\T}\prod_{\{i,j\}\in\T}W_{\{i,j\}}e^{u_i+u_j}$.
\begin{remark} Remark that usually
a result like ii) makes use of de Finetti's theorem: here, we provide a direct proof exploiting the explicit form of the density.
In Section 5, we apply Theorem \ref{annealed} and Theorem \ref{meas} i) ii) to give a new proof of Diaconis-Coppersmith formula including its de Finetti part.  
\end{remark}
\begin{remark}
The fact that (\ref{density}) is a density is not at all obvious. Our argument is probabilistic: (\ref{density}) is the law of the
random variables $(U_i)$. It can also be explained directly as a consequence of supersymmetry, see (5.1) in \cite{dsz}. The fact that
the measure (\ref{density}) normalizes at 1 is a fundamental property, which plays a crucial role in the localization and delocalization results of Disertori and Spencer
\cite{ds,dsz}.
\end{remark}
\begin{remark} \label{r3}
ii) implies that the VRJP $(Y_s)$ is a mixture of Markov jump processes. More precisely, let $(U_i)_{i\in V}$ be a random variable 
distributed according to 
(\ref{density}) and, conditionally on $U$, $Z$ be the Markov jump process with jump rates from $i$ to $j$ given by 
$\demi W_{i,j} e^{U_j-U_i}$. Then the time changed process $(Z_{B^{-1}}(s))_{s\ge 0}$ with
$$ B(t)= \sum_{i\in V} \sqrt{1+l_i^Z(t)}-1,
$$
where $(l_i^Z(t))$ is the local time of $Z$ at time $t$, has the law of the VRJP $(Y_s)$ with conductances $W$.
\end{remark}  

%\begin{corollary}
%\label{cd}
%*****Coppersmith-Diaconis measure*****
%\end{corollary}

\section{Proof of the Proposition \ref{pconv} and Theorem \ref{meas}}
\subsection{Proof of Proposition \ref{pconv}}

\renewcommand{\e}{\varepsilon}
\renewcommand{\te}{\tilde{\varepsilon}}

By a slight abuse of notation, we also use notation $L(T)$ for the $N\times N$ matrix $M(W^T,T)$ of that operator in the canonical basis. Let $\1$ be the $N\times N$ matrix with coefficients equal to $1$, i.e. $\1_{i,j}=1$ for all $i$, $j$ $\in V$, and let $I$ be the identity matrix. 

Let us define, for all $T\in \R^V$, 
\begin{eqnarray}\label{QT}
Q(T):=-\int_0^\iy \left(e^{uL(T)}-\frac{\1}{N}\right)\,du,
\end{eqnarray}
which exists since $e^{uL(T)}$ converges towards $\1/N$ at exponential rate.

Then $Q(T)$ is a solution of the Poisson equation for the Markov Chain $L(T)$, namely
$$L(T)Q(T)=Q(T)L(T)=I-\frac{\1}{N}.$$

Observe that $L(T)$ is symmetric, and thus $Q(T)$ as well.  

For all $T\in\R^V$ and $i$, $j$ $\in V$, let $E_i^T(\tau_j)$ denote the expectation of  the first hitting time of site $j$ for the continuous-time process with generator $L(T)$. Then
$$Q(T)_{i,j}=\frac{1}{N}E_i^T(\tau_j)+Q(T)_{j,j}$$
by the strong Markov property applied to (\ref{QT}).  As a consequence, $Q(T)_{j,j}$ is nonpositive for all $j$, using $\sum_{i\in V} Q(T)_{i,j}=0$.

Let us fix $l\in V$. We want to study the asymptotics of $T_l(t)-t/N$ as $t\to\iy$:
\begin{align}
\nonumber
T_l(t)-\frac{t}{N}&=\int_0^t \left(\1_{\{X_u=l\}}-\frac{1}{N}\right)\,du=\int_0^t (L(T(u))Q(T(u)))_{X_u,l}\,du\\
\nonumber &=\int_0^t \tL (Q(.)_{.,l})(X_u,T(u))\,du-\int_0^t \frac{\partial}{\partial T_{X_u}}Q(T(u))_{X_u,l}\,du\\
\label{poisson} &=Q(T(t))_{X_t,l}-Q(0)_{X_0,l}+M_l(t)-\int_0^t \frac{\partial}{\partial T_{X_u}}Q(T(u))_{X_u,l}\,du,
\end{align}
where $$M_l(t):=-Q(T(t))_{X_t,l}+Q(0)_{X_0,l}+\int_0^t \tL (Q(.)_{.,l})(X_u,T(u))\,du$$
is a martingale for all $l$. Recall that $\tL$ is the generator of $(X_t,T(t))$.

The following lemma shows in particular the convergence of $Q(T(t))_{k,l}$ for all $k$, $l$, as $t$ goes to infinity. It is a purely determistic statement, which does not depend on the trajectory of the process $X_t$ (as long as it only performs finitely many jumps in a finite time interval), but only on the added local time in $W^T$.
\begin{lemma}
For all $k$, $l$ $\in V$, $Q(T(t))_{k,l}$  converges as $t$ goes to infinity, and
$$\int_0^\iy\left|\frac{\partial}{\partial T_{X_u}}Q(T(u))_{X_u,l}\right|\,du<\iy.$$
\end{lemma}
\begin{proof}
For all $i$, $k$, $l$ $\in V$, let us compute 
$\frac{\partial}{\partial T_i}Q(T)_{k,l}$ : by differentiation of the Poisson equation, 
$$\frac{\partial}{\partial T_i}Q(T)_{k,l}=-\left(Q(T)\left(\frac{\partial}{\partial T_i}L\right)Q(T)\right)_{k,l}.$$
Now, for any real function $f$ on $V$, 
$$\frac{\partial}{\partial T_i}Lf(k)
=
\begin{cases}
\sum_{j\sim i}W_{i,j}^T(f(j)-f(i))&\tx{ if }k=i\\
W_{i,k}^T(f(i)-f(k))&\tx{ if }k\sim i, k\not=i\\
0&\tx{ otherwise.}
\end{cases}
$$
Hence
$$\frac{\partial}{\partial T_i}Lf(k)
=\sum_{j\sim i}W_{i,j}^T(f(j)-f(i))(\1_{\{i=k\}}-\1_{\{j=k\}})$$
and, therefore,
\begin{align}
\nonumber
\frac{\partial}{\partial T_i}Q(T)_{k,l}&=\sum_{j\sim i}W_{i,j}^T(Q(T)_{k,i}-Q(T)_{k,j})(Q(T)_{i,l}-Q(T)_{j,l})\\
&\label{derivq0}=\sum_{j\sim i}W_{i,j}^TQ(T)_{k,\nabla_{i,j}}Q(T)_{\nabla_{i,j},l}=\sum_{j\sim i}W_{i,j}^TQ(T)_{\nabla_{i,j},k}Q(T)_{\nabla_{i,j},l},
\end{align}
where we use the notation $f(\nabla_{i,j}):=f(j)-f(i)$ in the second equality, and the fact that $Q(T)$ is symmetric in the third one. 

In particular, for all $l\in V$ and $t\ge0$, 
\begin{equation}
\label{derivq}
\frac{d}{dt}Q(T(t))_{l,l}=\frac{\partial}{\partial T_{X_t}}Q(T(t))_{l,l}=\sum_{j\sim X_t}W_{X_t,j}\left(Q(T(t))_{\nabla_{X_t,j},l}\right)^2.
\end{equation}
Now recall that $Q(T(t))_{l,l}$ is nonpositive for all $t\ge0$; therefore it must converge, and 
$$\int_0^\iy \sum_{j\sim X_t}W_{X_t,j}\left(Q(T(t))_{\nabla_{X_t,j},l}\right)^2\,dt=(Q(T(\iy))-Q(0))_{l,l}<\iy.$$
The convergence of $Q(T(t))_{k,l}$ now follows from Cauchy-Schwarz inequality, using \eqref{derivq0}: for all $t\ge s$,
\begin{align*}
\left|(Q(T(t))-Q(T(s)))_{k,l}\right|
&=\int_s^t \sum_{j\sim X_u}W_{X_u,j}^TQ(T(u))_{\nabla_{X_u,j},k}Q(T(u))_{\nabla_{X_u,j},l}\,du\\
&\le\sqrt{(Q(T(t))-Q(T(s)))_{k,k}}\sqrt{(Q(T(t))-Q(T(s)))_{l,l}};
\end{align*}
thus $Q(T(t))_{k,l}$ is Cauchy sequence, which converges as $t$ goes to infinity. Now, using again Cauchy-Schwarz inequality,
\begin{align*}
&\int_0^\iy\left|\frac{\partial}{\partial T_{X_u}}Q(T(u))_{X_u,l}\right|\,du
\\
&=\int_0^\iy\left| \sum_{j\sim X_u}W_{X_u,j}^TQ(T(u))_{\nabla_{X_u,j},X_u}Q(T(u))_{\nabla_{X_u,j},l}\right| \,du\\
&\le\sqrt{\sum_{k\in V}(Q(T(\infty))-Q(T(0)))_{k,k}}\sqrt{(Q(T(\infty))-Q(T(0)))_{l,l}},
\end{align*}
which enables us to conclude.
\end{proof}
Next, we show that $(M_l(t))_{t\ge0}$ converges, which will complete the proof: indeed, this implies that the size of the jumps in that martingale goes to $0$ a.s., and therefore, by (\ref{poisson}), that $Q(T(t))_{X_t,l}$ must converge as well, again by (\ref{poisson}).

Let us compute the quadratic variation of the 
martingale $(M_l(t))_{t\ge0}$ at time $t$:
\begin{align*}
&\left(\frac{d}{d\e}\Es\left((M_l(T(t+\e))-M_l(t))^2|\F_t\right)\right)_{\e=0}\\
&=\left(\frac{d}{d\e}\Es\left((Q(T(t+\e))_{X_{t+\e},l}-Q(T(t))_{X_{t},l})^2|\F_t\right)\right)_{\e=0}\\
&=R(T(t))_{X_t,l}
\end{align*}
where, for all $(i,l,T)\in V\times V\times\R^V$, we let  
$$R(T)_{i,l}:=\tL(Q^2(.)_{.,l})(i,T)-2Q(T)_{i,l}\tL(Q(.)_{.,l})(i,T);$$
here $Q^2(T)$ denotes the matrix with coefficients $(Q(T)_{i,j})^2$, rather than $Q(T)$ composed with itself.
But
\begin{align*}
\tL(Q^2(.)_{.,l})(i,T)&=2\left(Q(T)\right)_{i,l}\left(\frac{\partial}{\partial T_{i}}Q(T)\right)_{i,l}+\left(L(T)Q^2(T)_{.,l}(i)\right)_{i,l}\\
Q(T)_{i,l}\tL(Q(.)_{.,l})(i,T)&=\left(Q(T)\right)_{i,l}\left(\frac{\partial}{\partial T_{i}}Q(T)\right)_{i,l}+Q(T)_{i,l}(L(T)Q(T)_{.,l}(i))_{i,l},
\end{align*}
so that
\begin{align*}
R(T)_{i,l}&=L(T)(Q^2(T)_{.,l})_{i,l}-2Q(T)_{i,l}(L(T)Q(T)_{.,l})_{i,l}\\
&=\sum_{j\sim i}W_{i,j}^T\left((Q(T)_{j,l})^2-(Q(T)_{i,l})^2\right)
-2Q(T)_{i,l}\sum_{j\sim i}W_{i,j}^T\left(Q(T)_{j,l}-Q(T)_{i,l}\right)\\
&=\sum_{j\sim i}W_{i,j}^T\left(Q(T)_{\nabla_{i,j},l}\right)^2=\frac{\partial}{\partial T_i}Q(T)_{l,l},
\end{align*}
using \eqref{derivq0} in the last equality. Thus
$$<M_l,M_l>_\iy=\int_0^\iy\frac{d}{du}Q(u)_{l,l}\,du=Q(T(\iy))_{l,l}-Q(0)_{l,l}\le-Q(0)_{l,l}<\iy.$$
Therefore $(M_l(t))_{t\ge0}$ is a martingale bounded in $L^2$, which converges a.s.
\begin{remark}
Once we know that $T_i(t)-t/N$ converges, then $T_i(\infty)=\infty$ for all $i\in V$, hence $Q(T(\iy))_{l,l}=0$, 
and the last inequality is in fact an equality, i.e.  $<M_l,M_l>_\iy=-Q(0)_{l,l}$.
\end{remark}

\subsection{Proof of Theorem \ref{meas} i)}
We consider, for $i_0\in V$, $T\in \R^V$, $\lambda \in \hhh_0$
\begin{eqnarray}\label{Psi}
\Psi(i_0, T, \lambda)=\int_{\hhh_0} e^{u_{i_0}} e^{i<\lambda, u>} \phi(W^T,u) du,
\end{eqnarray}
where
\begin{eqnarray}
\label{notphi}
\phi(W^T,u)= e^{-H(W^T,u)} \sqrt{ D(W^T,u)},
\end{eqnarray}
and $W^T_{i,j}= W_{i,j} e^{T_i+T_j}$.
We will prove that
$$
{1\over \sqrt{2\pi}^{N-1}} \Psi(i_0, T, \lambda) =\E_{i_0,T}\left( e^{i<\lambda, U>}\right),
$$
for all $i_0\in V$, $T\in \R^V$.
\begin{lemma}
The function $\Psi$ is solution of the Feynman-Kac equation
$$
i\lambda_{i_0} \Psi(i_0, T , \lambda) +(\tilde L \Psi)(i_0,T,\lambda)=0.
$$
\end{lemma}
\begin{proof}
Let $\overline T_i = T_i -{1\over N} \sum_{j\in V} T_j$. With the change of variables
$ \tilde u_i= u_i+\overline T_i$, we obtain
\begin{eqnarray}\label{Psitilde}
\Psi(i_0, T, \lambda)=\int_{\hhh_0} e^{\tilde u_{i_0}-\overline T_{i_0}} e^{i<\lambda, \tilde u-\overline T>} \phi(W^T,\tilde u- \overline T) d \tilde u
\end{eqnarray}
Remark now that $H(W^T, \tilde u -\overline T)= H(W^T, \tilde u -T)$ since $H(W^T,u)$ only depends on the differences $u_i-u_j$.
We observe that the coefficients of the matrix $M(W^T,u)$ only contain terms of the form $W_{i,j}e^{u_i+T_i+u_j+T_j}$, hence
$$
\sqrt{ D(W^T, \tilde u -\overline T)}= e^{{N-1\over N} \sum_j T_j} \sqrt{D(W,\tilde u)}.
$$
Finally, $<\lambda, \overline T>=<\lambda, T>$ since $\lambda\in \hhh_0$. This implies that
 \begin{eqnarray}\label{Psitilde2}
\Psi(i_0, T, \lambda)=\int_{\hhh_0} e^{\sum_j T_j} e^{\tilde u_{i_0}- T_{i_0}} e^{i<\lambda, \tilde u- T>} e^{-H(W^T,\tilde u -T)}  \sqrt{D(W,\tilde u)}  d\tilde u.
\end{eqnarray}
We have
\begin{eqnarray*}
&&{\partial \over \partial T_{i_0}} H(W^T, \tilde u -T)
\\
&=&
{\partial \over \partial T_{i_0}}\left( 2 \sum_{\{i,j\} \in E } W_{i,j} e^{T_i+T_j}\sinh^2\left( \demi(\tilde u_i-\tilde u_j -T_i+T_j)\right)\right)
\\
&=&
2\sum_{j\sim i_0} W_{i_0,j} e^{T_{i_0}+T_j}\left( \sinh^2\left(\demi( \tilde u_{i_0}-\tilde u_j -T_{i_0}+T_j)\right) -\demi\sinh\left(\tilde u_{i_0}-\tilde u_j -T_{i_0}+T_j\right)\right)
\\
&=&
\sum_{j\sim i_0} W_{i_0,j} e^{T_{i_0}+T_j}\left( e^{-\tilde u_{i_0}+\tilde u_j +T_{i_0}-T_j} -1\right)
\\
&=&
e^{-(\tilde u_{i_0}-T_{i_0})}L(T)(e^{\tilde u-T})(i_0).
\end{eqnarray*}
Hence,
\begin{eqnarray*}
&&-{\partial \over \partial T_{i_0}}\Psi(i_0, T, \lambda)
\\
&=&
\int_{\hhh_0}\left(i\lambda_{i_0}e^{\tilde u_{i_0}-T_{i_0}}+L(T)(e^{\tilde u-T})(i_0)\right) e^{\sum_j T_j} e^{i<\lambda, \tilde u-T>} e^{-H(W^T,\tilde u- T)} \sqrt{ D(W,\tilde u)} d\tilde u
\\
&=&
i\lambda_{i_0} \Psi(i_0,T,\lambda) +(L(T)\Psi)(i_0,T,\lambda).
\end{eqnarray*}
This gives
\begin{eqnarray*}
(\tilde L \Psi)(i_0, T, \lambda) = -i\lambda_{i_0} \Psi(i_0,T,\lambda).
\end{eqnarray*}

\end{proof}
Since $\Psi$ is a solution of the Feynman-Kac equation we deduce that for all $t>0$, $i_0\in V$, $\la\in  \hhh_0$, $T\in \R^V$, 
$$
\Psi(i_0,T,\lambda)=\E_{i_0,T}\left( e^{i<\lambda, \overline T (t)>} \Psi(X_t, T(t), \lambda)\right),
$$
where we recall that $\overline T_i (t)= T_i(t)-t/N $. 
Let us now prove that $\Psi(X_t, T(t), \lambda)$ is dominated and that
$\P_{i_0}$
a.s.
\begin{eqnarray}\label{cv}
\lim_{t\to \infty}  \Psi(X_t, T(t), \lambda) =\sqrt{2\pi}^{N-1}.
\end{eqnarray}
By the matrix-tree theorem, we have, denoting by $\ttt $ the set of spanning trees of $G$, and using again notation $\phi$ in \eqref{notphi},
\begin{eqnarray}\label{domination}
\nonumber e^{u_{i_0}} \phi(W^T,u)
&=&
e^{u_{i_0}} e^{-H(W^T,u)} \sqrt{ \sum_{\Lambda \in \ttt} \prod_{\{i,j\}\in \Lambda} W^T_{i,j} e^{u_i+u_j}}
\\
\nonumber &\le&
e^{N\max_{i\in V} \vert u_i\vert }e^{-\demi\sum_{\{i,j\} \in V} W^T_{i,j} (u_i-u_j)^2}\sqrt{ D(W^T,0)}
\\
&\le &
\left( \sum_{i\in V}  e^{N u_i}+e^{-Nu_i}\right) e^{-\demi\sum_{\{i,j\} \in V} W^T_{i,j} (u_i-u_j)^2}\sqrt{ D(W^T,0)}
\end{eqnarray}
This is a gaussian integrand: for any real $a$ and $i_0\in V$,
\begin{eqnarray*}
&&\int_{\hhh_0} e^{a u_{i_0}} e^{-\demi\sum_{\{i,j\} \in V} W^T_{i,j} (u_i-u_j)^2}\sqrt{ D(W^T,0)} du
\\
&=&
e^{-{\demi a^2} Q(T)_{i_0,i_0}}\int e^{-\demi<U-aQ(T)_{i_0,\cdot}, L(T)(U-aQ(T)_{i_0,\cdot})>}\sqrt{ D(W^T,0)} du
\\
&=& 
e^{-{\demi a^2} Q(T)_{i_0,i_0}} (2\pi)^{(N-1)/2}.
\end{eqnarray*}
where $Q(T)$ is defined at the beginning of Section 4.1.
Therefore for all $i_0\in V$, $(T_i)\in \R^V$
$$
\vert \Psi(i_0,T,\lambda)\vert \le 2 \sum_{i\in V}  (2\pi)^{(N-1)/2} e^{-\demi N^2 Q(T)_{i,i}},
$$
Using \eqref{derivq}, $Q(T(t))_{i,i}$ increases in $t$, hence
$$\vert \Psi(X_t, T(t), \lambda)\vert \le 2 \sum_{i\in V}  (2\pi)^{(N-1)/2} e^{-\demi N^2 Q(0)_{i,i}},$$
for all $t\ge 0$. Let us prove now (\ref{cv}). We have
\begin{eqnarray*}
&&\Psi(X_t, T(t), \lambda) 
\\
&=&
\int e^{i<\lambda,u>} e^{u_{X_t}} e^{-2\sum_{\{i,j\} \in E} W^{T(t)}_{i,j} \sinh^2(\demi(u_i-u_j))}\sqrt{D(W^{T(t)},u)} du
\\
&=&
\int e^{i<\lambda,u>} e^{u_{X_t}} e^{-2\sum_{\{i,j\} \in E} e^{2t/N} W^{\overline T(t)}_{i,j} \sinh^2(\demi(u_i-u_j))}\sqrt{D(W^{\overline T(t)},u)} e^{(N-1)t/N}du.
\end{eqnarray*}
Changing to variables $\tilde u_i=e^{t/N} u_i$, we deduce that $\Psi(X_t, T(t), \lambda)$ equals
\begin{eqnarray*}
\int e^{i<\lambda, e^{-t/N}\tilde u>} e^{e^{-t/N} \tilde u_{X_t}} e^{-2\sum_{\{i,j\} \in E} W^{\overline T(t)}_{i,j} e^{2t/N} \sinh^2(\demi e^{-t/N} (\tilde u_i-\tilde u_j))}
\sqrt{D(W^{\overline T(t)},e^{-t/N} \tilde u)} d\tilde u.
\end{eqnarray*}
Since $\lim_{t\to \infty} \overline T_i(t)=U_i$, the integrand converges pointwise to the Gaussian integrand
$$
e^{-\demi \sum_{\{i,j\} \in V} W^{U}_{i,j} (\tilde u_i-\tilde u_j)^2}
\sqrt{D(W^{U},0 )},
$$
whose integral is $\sqrt{2\pi}^{N-1}$.
Consider $\overline U_i=\sup_{t\ge 0} \overline T_i(t)$ and $\underline U_i= \inf_{t\ge 0} \overline T_i(t)$. Proceeding as in
(\ref{domination}) the integrand is dominated for all $t$ by
\begin{eqnarray*}
&&e^{Ne^{-t/N} \max_{i\in V} \vert \tilde u_i\vert }e^{-\demi \sum_{\{i,j\} \in V} W^{\overline T(t) }_{i,j} (\tilde u_i-\tilde u_j)^2}
\sqrt{D(W^{\overline T(t)},0 )}
\\
&\le &
\left( \sum_{i\in V}  e^{N \tilde u_i}+e^{-N\tilde u_i}\right)e^{-\demi \sum_{\{i,j\} \in V} W^{\underline U}_{i,j} (\tilde u_i-\tilde u_j)^2}
\sqrt{D(W^{\overline U},0 )}.
\end{eqnarray*}
which is integrable, which yields  (\ref{cv}) by dominated convergence.

\subsection{Proof of Theorem \ref{meas} ii)}
The same change of variables as in  (\ref{Psitilde2}), 
applied to $T_i=\log \lambda_i$, implies that, for any $j_0\in V$ and $(\lambda_i)_{i\in V}$ positive reals, 
$$
\frac{\prod_{i \in V} \lambda_i}{\sqrt{2\pi}^{N-1}}e^{u_{j_0}-\log(\lambda_{j_0})}e^{-\demi\sum_{\{i,j\}\in E} W_{i,j} \lambda_i\lambda_j\left(e^{\demi(u_j-u_i)}\sqrt{{\lambda_i\over\lambda_j}}
-e^{\demi(u_j-u_i)}\sqrt{{\lambda_j\over\lambda_i}}\right)^2} \sqrt{D(W,u)}
$$
is the density of a probability measure, which we call $\nu^{\lambda,j_0}$ (using that \eqref{density} defines a probability measure). 
Remark that this density  can be rewritten as
$$
{\prod_{i\in V} \lambda_i \over \sqrt{2\pi}^{N-1}}e^{u_{j_0}-\log(\lambda_{j_0})}e^{-\demi\sum_i\sum_{j\sim i} W_{i,j}(\lambda_i^2e^{u_j-u_i}- \lambda_i\lambda_j) } \sqrt{D(W,u)}
$$
Let  $(U_i)$ be a random variable distributed according to (\ref{density}), and, conditionally on $U$, let $(Z_t)$ be the Markov jump process 
starting  at $i_0$, and with jump rates  from $i$ to $j$
 $$
 \demi W_{i,j}e^{U_j-U_i}.
 $$
 Let $(\fff^Z_t)_{t\ge 0}$ be the filtration generated by $Z$, and let $E^U_i$ be the law of
 the process $Z$ starting at $i$, conditionally on $U$. 
 
 We denote by
 $(l_i(t))_{i\in V}$ the vector of local times of the process $Z$ at time $t$, and consider the positive continuous
additive functional of $Z$
$$
B(t)=\int_0^t \demi{1\over \sqrt{ 1+l_{Z_u}(u)}}du=\sum_{i\in V} \left(\sqrt{ 1+l_i(t)}-1\right),
$$
and the time changed process
$$
\tilde Y_s =Z_{B^{-1}(s)}.
$$

 Let us first prove that the law of $U$ conditioned on $\fff^Z_t$ is
 \begin{eqnarray}\label{Ucond}
 \lll(U | \fff_t^Z) = \nu^{\lambda(t), Z_t},
 \end{eqnarray}
 where $\lambda_i(t)= \sqrt{1+ l_i(t)}$.
 Indeed, let $t>0$: if $\tau_1, \ldots, \tau_{K(t)}$ denote the jumping times of the Markov process $Z_t$ up to time 
 $t$, then for any positive test function,
 \begin{align*}
 &E^{U}_{i_0}\left( \psi(\tau_1, \ldots, \tau_{K(t)}, Z_{\tau_1}, \ldots, Z_{\tau_{K(t)}})\right)=
 \\
 &\sum_{k=0}^\infty \sum_{i_1, \ldots, i_k} (\prod_{l=0}^{k-1} \demi W_{i_l, i_{l+1}})
 \int_{[0,t]^k} \psi((t_j), (i_j))e^{U_{i_k}-U_{i_0}}e^{-\demi \sum_{l=0}^{k-1} \left(\sum_{j\sim i_l} W_{i_l,j} e^{U_j-U_{i_l}}\right) (t_{l+1}-t_l)}  dt_1\cdots dt_k
 \end{align*}
 with the convention $t_{k+1}=t$.
 Hence, for any test function $G$,
 \begin{eqnarray*}
\E \left( G(U)| \fff^Z_t\right)
 =
 {\int_{\hhh_0} G(u) e^{u_{Z_t}}e^{-H(W,u)-\demi\sum_{i\in V} \left(\sum_{j\sim i} W_{i,j} e^{u_j-u_i}\right)l_i(t)}\sqrt{D(W,u)} du\over
  \int_{\hhh_0} e^{u_{Z_t}}e^{-H(W,u)-\demi\sum_{i\in V} \left(\sum_{j\sim i} W_{i,j} e^{u_j-u_i}\right)l_i(t)}\sqrt{D(W,u)} du}.
 \end{eqnarray*}
 Using that we can write $H(W,u)= \demi \sum_{i\in V}\sum_{j\sim i} (e^{u_j-u_i}-1)$, and introducing adequate constants in the numerator
 and denominator we have 
 \begin{eqnarray*}
&&\E \left( G(U)| \fff^Z_t\right) 
\\
&=&
{\sqrt{2\pi}^{-(N-1)} \int_{\hhh_0} G(u) (\prod \lambda_i) e^{u_{Z_t}-\log\lambda_{Z_t}}e^{-\demi\sum_{i}\sum_{j\sim i} W_{i,j}(\lambda_i(t)^2e^{ u_j-u_i}- \lambda_i(t)\lambda_j(t))}\sqrt{D(W,u)} du\over
\sqrt{2\pi}^{-(N-1)}\int_{\hhh_0} (\prod \lambda_i) e^{u_{Z_t}-\log\lambda_{Z_t}}e^{-\demi\sum_{i}\sum_{j\sim i} W_{i,j}(\lambda_i(t)^2e^{ u_j-u_i}- \lambda_i(t)\lambda_j(t))}\sqrt{D(W,u)} du}
 \end{eqnarray*}
 (recall that
 $\lambda_i(t)=\sqrt{1+l_i(t)}$). The denominator is 1 since it is the integral of the density of $\nu^{\lambda(t), Z_t}$. This proves  (\ref{Ucond}).
 
 Subsequently, by (\ref{Ucond}), conditioned on $(\fff_t^Z)$, if the process $Z$ is at $i$ at time $t$, then it
 jumps to a neighbour $j$ of $i$ with rate
 $$
 \demi W_{i,j} \E^{\nu^{\lambda(t),i}}\left( e^{U_j-U_i}\right)
 =
 \demi W_{i,j} {\lambda_j(t)\over \lambda_i(t)}.
 $$
In order to conclude, we now compute the corresponding rate for  $\tilde Y$: by definition,
 $$
 B'(t)=\demi{1\over \sqrt{ 1+l_{Z_t}(t)}}.
 $$
 Therefore, similarly as in the proof of Lemma \ref{invtime},
 \begin{eqnarray*}
 \P\left( \tilde Y_{s+ds}=j | \fff_t^Z\right)&=&
 \P\left(Z_{B^{-1}(s+ds)}=j | \fff_t^Z\right)
\\
&=&
\demi W_{Y_s,j} {1\over B'(B^{-1}(s))} {\lambda_j(B^{-1}(s)) \over \lambda_{Y_s}(B^{-1}(s))}\,ds
\\
&=&
W_{Y_s,j} \lambda_j(B^{-1}(s))\,ds.
\end{eqnarray*} 
Let $(\tl_i(s))$ be the local time of the process $\tilde Y$. Then
\begin{eqnarray*}
(\tl_i(B(t))'=
B'(t) \indic_{\{\tilde Y_{B(s)=i}\}}=\demi (1+l_i(t))^{-\demi}\indic_{\{Z_t=i\}}.
\end{eqnarray*}
This implies
\begin{equation}
\label{ltl}
\tl_i(B(t))= \sqrt{1+l_i(t)} -1
\end{equation}
 and
$$ \P\left( \tilde Y_{s+ds}=j | \fff_t^Z\right)=
W_{\tilde Y_s,j} (1+\tl_j(s))\,ds$$
This means that the annealed law of $\tilde Y$ is the law of a VRJP with conductances $(W_{i,j})$ (this is the content of remark \ref{r3}).

Therefore, the process defined, for all $t\ge0$, by $\tilde Y_{A^{-1}(t)}=Z_{(A\circ B)^{-1}(t)}$, is equal in law to $(X_t)_{t\ge0}$; let us denote by $T$ its local time, and show that 
$T_i(t)-t/N$
converges to $U_i$ as $t\to\iy$, which will complete the proof.

First note, using \eqref{tlx} and \eqref{ltl}, that, for all $i\in V$, 
$$T_i((A\circ B)(t))=\log(\tl_i(B(t))+1)=\log(1+l_i(t))/2.$$

On the other hand, conditionally on $U$, the Markov Chain $Z$ has invariant measure $(Ce^{2U_i})_{i\in V}$, $C:=(\sum_{i\in V}e^{2U_i})^{-1}$, so that $l_i(t)/(Ce^{2U_i}t)$ converges to $1$ as $t\to\iy$, for all $i\in V$. 

Therefore, for all $i\in V$, 
$$T_i(t)-T_{i_0}(t)=\frac{1}{2}\log\left(\frac{1+l_i((A\circ B)^{-1}(t))}{1+l_i((A\circ B)^{-1}(t))}\right),$$
which converges towards $U_i-U_{i_0}$ as $t\to\iy$, which enables us to conclude.
\section{Back to Diaconis-Coppersmith formula}\label{BCP}

It follows from de Finetti's theorem for Markov chains \cite{diaconis-freedman,rolles1} that the law of the ERRW is a mixture of reversible Markov chains; its mixing measure was explicitly described by Coppersmith and Diaconis (\cite{coppersmith}, see also \cite{keane-rolles1}). 

Theorems 1 and 2 enable us to retrieve this so-called Coppersmith-Diaconis formula, including its de Finetti part: they imply that the ERRW $(X_n)_{n\in\N}$
follows the annealed law of a reversible Markov chain in a random conductance network $x_{i,j}=W_{i,j} e^{U_{i}+U_j}$ where
$W_e\sim Gamma(a_e,1)$, $e\in E$,  are independent random variables and, conditioned on $W$, the random variables $(U_i)$
are distributed according to the law (\ref{density}). 

Let us compute the law it induces on the random variables $(x_e)$.
The random
variable $(x_e)$ is only significant up to a scaling factor, hence
we consider a 0-homogeneous bounded measurable test function $\phi$; by Theorem 2,
\begin{eqnarray*}
&&\E\left( \phi((x_e))\right)
\\
&=&
{1\over \sqrt{ 2\pi}^{N-1}}\int_{\R_+^E\times \hhh_0} \phi(x) \left(\prod_{e\in E} {1\over \Gamma(a_e)}W_e^{a_e}e^{-W_e} \right) e^{u_{i_0}} \sqrt{D(W,u)}e^{-H(W,u)}{dW\over W}  du
\end{eqnarray*}
where we write ${dW\over W}=\prod_{e\in E}{dW_e\over W_e}$.
Changing to coordinates $\overline u_i=u_i-u_{i_0}$ yields
\begin{eqnarray*}
C(a)\int_{\R_+^E\times \R^{V\setminus\{i_0\}}} \phi(x) \left(\prod_{e\in E} W_e^{a_e}e^{-W_e} \right) 
e^{-\sum_{i\neq i_0} \overline u_i} \sqrt{D(W, \overline u)}e^{-H(W,\overline u)} {dW\over W} d\overline u
\end{eqnarray*}
with $d\overline u=\prod_{i\neq i_0} d\overline u_i$ and
 $C(a)= {1\over \sqrt{ 2\pi}^{N-1}}\left(\prod_{e\in E} {1\over \Gamma(a_e)}\right)$.
But
  $$
  -\sum_{e\in E} W_e -H(W,\overline u)= -\demi \sum_{\{i,j\}\in E} W_{i,j}e^{\overline u_i+\overline u_j}\left(e^{-2\overline u_j}+e^{-2\overline u_i}\right).
  $$
  The change of  variables 
  $$\left((x_{i,j}=W_{i,j}e^{\overline u_i+\overline u_j})_{\{i,j\}\in E}, (v_i=e^{-2\overline u_i})_{i\in V\setminus\{i_0\}}\right),
  $$
with $v_{i_0}=1$  implies
$$
  -\sum_{e\in E}W_e -H(W,\overline u)= -\demi \sum_{i\in V } v_i x_i,
$$
where $x_i=\sum_{j\sim i} x_{i,j}$, and $\E\left( \phi((x_e))\right)$ is equal to the integral
\begin{eqnarray*}
C'(a)\int \phi(x) \left(\prod_{e\in E} x_e^{a_e} \right) \left( \prod_{i\in V} v_i^{(a_i+1)/2}\right) v_{i_0}^{-\demi}  \sqrt{D(x)}
e^{-\demi \sum_{i\in V } v_i x_i} \left(\prod_{e\in E} {dx_e\over x_e}\right) \left( \prod_{i\neq i_0} {dv_i\over v_i}\right),
\end{eqnarray*}
with $a_i=\sum_{j\sim i} a_{i,j}$, $D(x)$  determinant of any diagonal minor of the $N\times N$ matrix
$$
m_{i,j}=\left\{ \begin{array}{ll} x_{i,j}  & \hbox{ if $i\neq j$}
\\
-\sum_{k\sim i} x_{i,k} &\hbox{ if $i=j$}
\end{array}\right.
$$
 and
$$
C'(a)= {2^{-N+1}\over \sqrt{ 2\pi}^{N-1}}\left(\prod_{e\in E} {1\over \Gamma(a_e)}\right).
$$
Let $e_0$ be a fixed edge: we normalize the conductance to be 1 at $e_0$ by changing to variables  
$$\left(\left(y_e={x_e\over x_{e_0}}\right)_{e\neq e_0}, \left( z_i=x_{e_0} {v_i}\right)_{i\in V}\right),
$$ 
with $y_{e_0}=1$. Now, observe that
$$
\left(\prod_{e\in E} {dx_e\over x_e}\right) \left( \prod_{i\neq i_0} {dv_i\over v_i}\right)= \left(\prod_{e\in E, e\neq e_0} {dy_e\over y_e}\right) 
\left( \prod_{i \in V} {dz_i\over z_i}\right).
$$
We deduce that $\E\left( \phi((x_e))\right)$ equals the integral
  \begin{eqnarray*}
C(a)\int_{\R_+^V\times \R_+^{E\setminus\{e_0\}}} \phi(y) \left(\prod_{e\in E} y_e^{a_e} \right) \left( \prod_{i\in V} z_i^{a_i/2}\right) z_{i_0}^{-\demi}  \sqrt{D(y)}
e^{-\demi \sum_{i\in V } z_i y_i} \left({dy\over y}\right)\left({dz\over z}\right),
\end{eqnarray*}
with ${dy\over y}= \prod_{e\neq e_0} {dy_e\over y_e}$ and ${dz\over z}=\prod_{i\in V} {dz_i\over z_i}$.
Therefore, integrating over the variables $z_i$
\begin{eqnarray*}
\E\left( \phi((x_e))\right)
=
C''(a)\int_{\R_+^{E\setminus\{e_0\}}} \phi(y) y_{i_0}^{\demi}  {\left(\prod_{e\in E} y_e^{a_e} \over \prod_{i\in V} y_i^{(a_i+1)/2}\right)}  \sqrt{D(y)}
\left({dy\over y}\right),
\end{eqnarray*}
where
$$
C''(a)={2^{1-N-\sum_{e\in E} a_e}\over \pi^{(N-1)/2}}{\Gamma(a_{i_0}/2) \prod_{i\neq i_0} \Gamma((a_i+1)/2)\over \prod_{e\in E} \Gamma(a_e)}
$$
which is Diaconis-Coppersmith formula: the extra term $(\vert E\vert -1)!$  in
\cite{keane-rolles1,diaconis-rolles}  arises from the normalization of  $(x_e)_{e\in E}$ on the simplex
$\Delta=\{\sum x_e=1\}$ (see Section 2.2 \cite{diaconis-rolles}).

\section{The supersymmetric hyperbolic sigma model}

We first relate  VRJP to the supersymmetric hyperbolic sigma model studied in Disertori, Spencer and Zirnbauer \cite{dsz,ds}. 
For notational purposes, we restrict our attention to the $d$-dimensional lattice, that is, our graph is $\Z^d$ with $x\sim y$ if $\vert x-y\vert_1 =1$.
We denote by $E$ the set of edges $E=\{\{i,j\}, \; i\sim j\}$. For a subset $\Lambda \subset \Z^d$ we denote by $E_{\Lambda}$ the set
of edges with both extremities in $\Lambda$.

We start by a description of the measures defined in \cite{dsz,ds}. 
Let $V\subset \Z^d$ be a connected finite subset containing $0$.  Let $\beta_{i,j}$, $i,j\in V$, $i\sim j$ be some positive weights on the edges, and 
$\e=(\epsilon_i)_{i \in V}$ be a vector of non-negative reals, $\epsilon \neq 0$.
Let  $\mu_V^{\e,\be}$ be a generalization of the measure studied in \cite{ds} (see (1.1)-(1.7)  in that paper), namely
\begin{align*}
d\mu_V^{\e,\be}(t):&=\left(\prod_{j\in V}\frac{dt_j}{\sqrt{2\pi}}\right)e^{-\sum_{k\in V}t_k}e^{-F_V^\beta(\nabla t)}e^{-M_V^{\e}(t)}\sqrt{\mathsf{det}\,\,A_V^{\e,\beta}}\\
&=\left(\prod_{j\in V}\frac{dt_j} {\sqrt{2\pi}}\right)e^{-F_V^\beta(\nabla t)}e^{-M_V^\e(t)}\sqrt{\mathsf{det}\,\,D_V^{\e,\beta}}
\end{align*}
where $A_V^{\e,\beta}=A^{\e,\beta}$ and $D_V^{\e,\beta}=D^{\e,\beta}$ are defined by, for all $i$, $j$ $\in V$, by
$$
A_{ij}^{\e,\beta}=e^{t_i}D_{ij}^{\e,\beta} e^{t_j}=
\begin{cases}
0&|i-j|>1\\
-\be_{ij}e^{t_i+t_j}&|i-j|=1\\
\sum_{l\sim i, l\in V}\be_{il}e^{t_i+t_l}+\e_i e^{t_i}&i=j\\
\end{cases}
$$
\begin{align*}
F_V^\beta(\nabla t)&:=\sum_{\{i,j\}\in E_V}\be_{ij}(\cosh(t_i-t_j)-1)\\
M_V^\e(t)&:=\sum_{i\in V}\e_i(\cosh t_i-1).
\end{align*}
The fact that $\mu_V^{\e,\be}$ is a probability measure can be seen as a consequence of supersymmetry (see (5.1) in \cite{dsz}). This is also a consequence of Theorem \ref{meas} i), cf later.

The measure $\mu_V^{\e,\be}$ is directly related to the measure (\ref{density}) defined in theorem \ref{meas} as follows.
Let add an extra point $\delta$ to $V$, $\tilde V=V\cup\{\delta\}$, and extra edges $\{i,\de\}$ connecting any site $i\in V$ to $\de$, i.e. $\tilde E_V=E_V\cup\cup_{i\in V}\{i,\de\}$. Consider the VRJP on $\tilde V$ starting at $\delta$ and
with conductances $W_{i,j}=\beta_{i,j}$, if $i\sim j$ in $V$, and $W_{i,\delta}=\epsilon_i$.
Let us again use notation $(U_i)_{i\in\tilde{V}}$ for the limiting centred occupation times of VRJP on
$\tilde V$ starting at $\delta$, and consider the change of variables, from $\hhh_0$ into $\R^V$, which maps $u_i$ to $t_i:=u_i-u_\de$. Then, by Theorem \ref{meas},  for any test function $\phi$, letting $\iota$ be the canonical injection $\R^V\longrightarrow\R^{\tilde{V}}$,
\begin{align*}
\Es^{W}_\de(\phi(U-U_\delta ))&=\frac{1}{(2\pi)^{\vert V\vert/2}}\int_{\hhh_0}\phi(u-u_\delta)e^{u_\de}e^{-H(W,u)} \sqrt{D(W,u)}\,du\\
&=\frac{1}{(2\pi)^{\vert V\vert/2}}\int_{\R^V}\phi(t)e^{-\sum_{i\in V}t_i}e^{-H(W,\iota(t))} \sqrt{D(W,\iota(t))}\left(\prod_{i\not=\de}\,dt_i\right)\\
&=\mu_V^{\e,\be}\left(\phi(t)\right),
\end{align*}
which means that $U-U_\delta$ is distributed according to $\mu_V^{\e,\be}$.
Indeed $A_V^\e$ is the restriction to $V\times V$ of the matrix $M(W,\iota(t))$ (which is on $\tilde{V}\times \tilde{V}$) (so that $\mathsf{det}\,\,A_V^\e=D(W,\iota(t))$), and 
$F_V(\nabla t)+M_V^\e(t)=H(W,\iota(t))$. 

We will be interested in the VRJP on finite subsets of $\Z^d$ starting at 0. In order to apply directly results of \cite{ds} we consider
the VRJP on $\Z^d$ with an extra point $\delta$ uniquely connected to 0 and with $W_{0,\delta}=\epsilon_0=1$,
$W_{i,j}=\beta_{i,j}$, $i\sim j$ in $\Z^d$.
Clearly, the trace on $\Z^d$ of the VRJP starting from $\delta$ has the law of the VRJP on $\Z^d$ starting from 0.
When $V$ contains $0$ 
%and when $(\e_i)=\epsilon\delta_{0}$ for some real $\epsilon>0$ (where $\delta_0$
%is the Dirac at 0) we write simply $d\overline \mu_V^{\e,\be}$
%for $d\mu_V^{\e \delta_0,\be}$.  
the limiting occupation time $U_i -U_\delta$ of the VRJP on $\tilde V=V\cup\{\delta\}$ starting
at $\delta$ is then distributed according to $d\mu_V^{\delta_0,\be}$, where $\delta_0$ is the Dirac at 0.

%Let us add a vertex $\de$ connected to $0$, that is, consider the graph with vertices $\Z^d\cup\{\delta\}$ with an extra edge $\{\de,0\}$. 

%Assume that we are given positive conductances on the network : in order to be closer to the notation
%of \cite{dsz,ds}, we denote by
%$\beta_{x,y}=W_{x,y}$ the conductances on the edge $\{x,y\}$, if $x,y \in \Z^d$ and $\e=W_{0,\delta}$ the conductance on the edge $\{0,\delta\}$.
%Note that VRJP on $\Z^d$ and on $\Z^d\cup \{\delta\}$ are easy to compare.

Set, for all $\be>0$,
$$I_\be:=\sqrt{\be}\int_{-\iy}^\iy\frac{dt}{\sqrt{2\pi}}e^{-\be(\cosh t-1)},
$$
which is strictly increasing in $\be$.
Let $\beta^d_c$ be defined as the unique solution to the equation
$$
I_{\beta^d_c}e^{\beta^d_c(2d-2)} (2d-1)=1
$$
for all $d>1$, $\beta^d_c:=\iy$ if $d=1$.

If the parameters $\beta_e$ are constant equal to $\beta$ over all edges $e$, then Theorem 2 in \cite{ds} readily implies that VRJP over any graph $\Z^d$ is recurrent for $\beta<\beta^d_c$ (i.e. for large reinforcement).
\begin{theorem}[Disertori and Spencer \cite{ds}, Theorem 2]\label{exp}
Suppose that $\beta_e=\beta$ for all $e$. Then there exists a constant $C_0:=2d/(2d-1)>0$ such that, for all $\Lambda\subset \Z^d$ finite connected containing 0, $x\in\La$,  $0<\beta <\beta^d_c$, 
$$
\mu_\Lambda^{\eta\delta_0,\beta}\left( e^{t_x/2} \right)\le C_0 I_{\eta}\left[ I_\beta e^{\beta(2d-2)}(2d-1)\right]^{\vert x\vert}.
$$
\end{theorem}
\begin{corollary}\label{rec}
For $0<\beta <\beta^d_c$, let $(Y_n)$ be the discrete time process associated with 
the VRJP on $\Z^d$ starting from 0 with constant conductance $\beta$. Then $(Y_n)$ is a mixture of reversible positive recurrent 
Markov chains.
\end{corollary}
\begin{proof}
We prove this for the VRJP on $\Z^d$ with an extra point $\delta$ connected to 0 only, and conductances
$W_{x,y}=\beta$ and $W_{0,\delta}=1$, which is clearly stronger. 
On a finite connected subset $V\subset \Z^d$ containing 0, 
we know from Theorem \ref{meas} that $(Y_n)_{n\in\N}$, the discrete-time process associated with 
$(Y_s)_{s\ge 0}$, is a mixture of reversible Markov chains with conductances $c_{x,y}=\beta e^{t_x+t_y}$, where $(t_x)_{x\in V}$
has law $\mu_V^{\delta_0,\beta}$. 

Now Theorem \ref{exp} implies that $\mu_V^{\delta_0,\beta}( (c_{e}/c_{\delta,0})^{1/4})$ decreases exponentially with the distance from $e$ to 0: indeed, by Cauchy-Schwarz inequality, 
\begin{eqnarray*}
\mu_V^{\delta_0,\beta}( (c_{x,y}/c_{\delta,0})^{1/4})&\le& \left[\mu_V^{\delta_0,\beta}(e^{t_x/2})\mu_V^{\delta_0,\beta}(e^{(t_y-t_0)/2})\right]^{1/2}
\\
&\le&
C\left[\mu_V^{\delta_0,\beta}(e^{t_x/2}) \mu_V^{\delta_0,\beta}(e^{\demi (\cosh(t_0)-1)}e^{t_y/2})\right]^{1/2}
\\
&\le& 
2 C\left[\mu_V^{\delta_0,\beta}(e^{t_x/2})\mu_V^{\delta_0/2,\beta}(e^{t_y/2})\right]^{1/2}
\end{eqnarray*}
for some $C>0$ such that  $\vert z \vert\le 4\log C + \cosh (z)-1$.
This implies that there exists constants $c_1>0$, $c_2>0$, such that $\mu_V^{\delta_0,\beta}((c_{x,y}/c_{\delta,0})>e^{-c_1 \vert x\vert})
\le e^{-c_2 \vert x\vert}$. Following the proof of lemma 5.1 of \cite{merkl-rolles4} it implies that $(Y_n)$ is a mixture of positive recurrent Markov chains.
\end{proof}

By Theorems \ref{annealed} and \ref{meas}, the ERRW with constant initial weights $a>0$ corresponds to the case where
$(\beta_{e})_{e\in E}$ are independent random variables
with $\hbox{Gamma} (a,1)$  distribution for some parameter $a>0$: we can infer a similar localization and recurrence result for $a$ small enough.
This requires an extension of Theorem \ref{exp}
 for random gamma weights $(\beta_e)_{e\in E}$: we propose one in the following Proposition \ref{expvrjp}, in the same line of proof as in \cite{ds}.

%\begin{proposition}
%\label{expvrjp}
%Let $\G_x$ be the set of non-intersecting paths from $0$ to a vertex $x$ in $V$. For all $x\in V$, 
%$$\mu_V^{\e,\be}\left(e^{t_x/2}\right)\le\sqrt{\e}\sum_{\g\in\G_x}e^{{\sum_{\{i,j\}\in E, i\in\La_\g, j\not\in\La_\g}\be_{ij}}}I_{\e}\prod_{e\in\g} I_{\be_e}.$$
%where $\La_\g$ and $\La_\g^c$ are respectively the set of vertices in the path and its complement.
%\end{proposition}
\begin{proposition}
\label{expvrjp}
Let $a$ be a positive real, and assume that the conductances $(\beta_{e})_{e\in E}$ are i.i.d. with law
Gamma$(a,1)$. Denote by $\E$ the expectation with respect to the random variables $(\beta_e)_{e\in E}$.
For all $d\ge 1$, there exists $a_c^d>0$, $\de\in(0,1)$, such that if $a<a_c^d$, there exists 
$C_0>0$ depending only on $a$ and $d$ such that for all $x\sim y$, 
$$
\E\left(\mu_\La^{\eta \delta_0,\be}(e^{t_x/2})\right)\le C_0 I_\eta\delta^{\vert x\vert}, \;\;\; 
\E\left((\beta_{x,y})^{{1\over 4}} \mu_\La^{\eta \delta_0,\be}(e^{t_x/2})\right)\le C_0I_\eta \delta^{\vert x\vert}
$$
independently on the finite connected subset $\La$ containing 0 and $x$, $y$.
\end{proposition}
\begin{corollary}\label{rec2}
The ERRW on $\Z^d$ starting at 0 with constant initial weight $a>0$ is a mixture of positive recurrent Markov chains for $a<a_c^d$ (where $a_c^d$ is defined in Proposition \ref{expvrjp}). 
%Let $\G_x$ be the set of non-intersecting paths from $0$ to a vertex $x$ in $V$. For all $x\in V$, 
%$$\mu_V^{\e,\be}\left(e^{t_x/2}\right)\le\sqrt{\e}\sum_{\g\in\G_x}e^{{\sum_{\{i,j\}\in E, i\in\La_\g, j\not\in\La_\g}\be_{ij}}}I_{\e}\prod_{e\in\g} I_{\be_e}.$$
%where $\La_\g$ and $\La_\g^c$ are respectively the set of vertices in the path and its complement.
\end{corollary}
\begin{remark}
Corollary \ref{rec2} also holds on any graph of bounded degree, and for possibly non-constant weights $(a_e)_{e\in E}$ with $a_e<a_c$ for some $a_c>0$. Indeed, the proof of Proposition \ref{expvrjp} also holds for independent (not necessarily i.i.d.) conductances $(\be_e)$ with $\Es(\sqrt{\be_e}(\log(1+\be_e^{-1}))$ sufficiently small, when the graph is of bounded degree.
\end{remark}
%We can then sum up the upper bound from this result over the random variables $(\beta_{e})_{e\in E}$, assuming they
%are random i.i.d. and $\E\left( e^{\beta_{x,y}}\right)<\infty$:  this implies recurrence of 
%VRJP in the i.i.d. random environment $\beta_e\sim\hbox{Gamma}(a,\mu)$ for any $\mu>1$ and $a$ small
%enough, but does not cover the case $\mu=1$ of the ERRW.
\begin{corollary}
\label{trans}
For any $d\ge3$, there exists $\be_c(d)$ such that, for all $\be>\be_c(d)$, VRJP on $\Z^d$ with constant conductance $\be$ is transient.
\end{corollary}
\begin{proof} (Proposition \ref{expvrjp})
The strategy is to follow the proof of \cite{ds}, Theorem 2, and to truncate the random variables $\beta_e$ at adequate positions. For convenience  we provide a self-contained proof but the only new input compared to \cite{ds}, Theorem 2, lies in the threshold argument (\ref{thres1}--\ref{thres3}).  
Let us define, for any $\La\subset \Z^d$,  $(\epsilon_i)_{i\in \La}\in\R_+^\La$
$$d\nu_\La^{\e ,\be}(t):=\left(\prod_{i\in \La}\frac{dt_j}
{\sqrt{2\pi}}\right)e^{-F^\beta_\La(\nabla t)}e^{-M_\La^{\e}(t)},$$
which is not a probability measure in general.

We fix now a finite connected subset $\Lambda\subset \Z^d$ containing 0, and $x$. 
Let $\G_x$ be the set of non-intersecting paths in $\Lambda$ from $0$ to $x$. For notational purposes, any element $\g$ in $\G_x$ is defined here as the set of non-oriented edges in the path. We let $\La_\g$ and $\La_\g^c$ be respectively the set of vertices in the path and its complement.  We say that an edge $e$ is adjacent to the path $\gamma$ if $e$ is not in $\gamma$
and has one adjacent vertex in $\gamma$, i.e. if $e=\{i,j\}$ with $i\in\La_\g, j\not\in\La_\g$; we write $e\sim \gamma$.

We first proceed similarly to (3.1)-(3.4) in \cite{ds}, Lemma 2. Let $\ttt_\Lambda$ be the set of spanning trees of $\Lambda$.
By the matrix-tree theorem
$$
\det(A_\Lambda^{\eta\delta_0,\beta})= \eta e^{t_0} \sum_{T\in \ttt_\Lambda} \prod_{\{i,j\}\in T} \beta_{\{i,j\}} e^{t_{i}+t_{j}}.
$$
In a spanning tree $T$ there is a unique path between 0 and $x\in \Lambda$.
Decomposing this sum depending on this path we deduce
$$
\det(A_\Lambda^{\eta\delta_0,\beta})=
 \eta e^{t_0} \sum_{\g\in\G_x} \left( \prod_{\{i,j\}\in \gamma} \beta_{\{i,j\}} e^{t_i+t_j}\right) \sum_{T'\in \ttt_\Lambda^\g} 
 \prod_{{\{i,j\}}\in T'} \beta_{\{i,j\}} e^{t_{i}+t_{j}}
$$
where $\ttt_\Lambda^\g$ is the set of subsets $T'\subset E_\Lambda\setminus \gamma$ such that $\gamma\cup T'$ is
a spanning tree.
By the matrix-tree theorem, we have
\begin{eqnarray}
\label{MTA}
\sum_{T'\in \ttt_\Lambda^\g} 
 \prod_{{\{i,j\}}\in T'} \beta_{\{i,j\}} e^{t_{i}+t_{j}}=
\det ( A_{\Lambda_\g^c}^{\epsilon, \beta})
\end{eqnarray}
where $(\epsilon_i)_{i\in \La_\g^c}$ is the vector defined by
$$\e_i:=\sum_{k\in\La_\g, k\sim i}\be_{\{i,k\}}e^{t_k}, \;\;\; \forall i\in \La_\g^c$$
It follows that
\begin{eqnarray}
\label{DDD}
\mathsf{det}\, D_\La^{\eta\delta_0,\beta}=\eta e^{-t_x}\sum_{\g\in\G_x}\left(\prod_{e\in\g}\be_{e}\right) \,\,\mathsf{det}\, D_{\La_\g^c}^{\e,\beta}.
\end{eqnarray}
 
  Let us define, similarly as in (2.12) and (2.14) in \cite{ds}, for $t_\g=t_{|\Lambda_\g}$ the restriction of
 $t$ to the vertices on the path $\gamma$,
\begin{align}
\label{zzz} Z_{\La_\g^c}^{\g,\beta}(t_{\g})&:=\nu_{\La_\g^c}^{\eta\delta_0,\beta}
\left(\sqrt{\mathsf{det}\, D_{\La_\g^c}^{\e,\beta}}e^{-F_{\partial\g}^\beta(\nabla t)}\right)\\
\nonumber F_{\partial\g}^\beta(\nabla t)&:=\sum_{k\in\La_\g, j\in\La_\g^c, k\sim j}\be_{kj}(\cosh(t_j-t_k)-1).
\end{align}

Now
\begin{align}
\nonumber
\mu_\La^{\eta\delta_0,\be}\left(e^{t_x/2}\right)&=\nu_\La^{\eta\delta_0,\be}\left(\sqrt{\mathsf{det}\, D_\La^{\eta\delta_0,\beta} e^{t_x}}\right)
=\sqrt{\eta}\nu_\La^{\eta\delta_0,\be}\left(\sqrt{\sum_{\g\in\G_x}\prod_{e\in\g}\be_e \,\,\mathsf{det}\, D_{\La_\g^c}^{\e,\beta}}\right)\\
\label{estmu}
&\le\sqrt{\eta} \sum_{\g\in\G_x}\left(\prod_{e\in\g}\sqrt{\be_e}\right)\nu_{\La_\g}^{\eta\delta_0,\beta}\left(Z_{\La_\g^c}^{\g,\beta}(t_\g)\right),
\end{align}
using (\ref{DDD}) in the second equality 
%that $\mathsf{det}_{\La_\g^c}\, D_\La^\e=\mathsf{det}\, D_{\La_\g^c}^{\te}$,
and, in the inequality, that for all $\g\in\G_x$, 
$$d\nu_\La^{\eta\delta_0,\beta}(t)=d\nu_{\La_\g}^{\eta\delta_0,\beta}(t)d\nu_{\La_\g^c}^{\eta\delta_0,\beta}(t)e^{-F_{\partial\g}(\nabla t)}.$$ 
 
The new argument compared to theorem \ref{exp} which allows to handle the case of random parameters $\beta$ is
the following truncation. Given $\g\in\G_x$, let $(\tilde\beta_e)$ be the set of random variables defined by
\begin{eqnarray}
\label{thres1}
\tilde\beta_e=\left\{\begin{array}{ll}
\min(\beta_e,1),&\hbox{ if  $e\sim \gamma$,}
\\
\beta_{e}, &\hbox{otherwise.}
\end{array}
\right.
\end{eqnarray}
First note that, trivially,
\begin{eqnarray}\label{thres2}
e^{-F^\beta_{\partial \g}(\nabla t)}\le e^{-F^{\tilde \beta}_{\partial \gamma}(\nabla t)}.
\end{eqnarray}
On the other hand, identity (\ref{MTA}) implies that
\begin{eqnarray}\label{thres3}
\det(D_{\La_\g^c}^{\e,\beta})\le \det(D_{\La_\g^c}^{\te,\tilde \beta})\left( \prod_{e\sim \gamma} \max(\beta_{e},1)\right),
\end{eqnarray}
where $(\tilde\epsilon_i)_{i\in \La_\g^c}$ is the vector defined by
$$\te_i:=\sum_{k\in\La_\g, i\sim k}\tilde\be_{\{i,k\}}e^{t_k}, \;\;\; \forall i\in \La_\g^c$$
(In the last argument we used that, for any $\{i,j\}$ adjacent to $\gamma$, $\beta_{i,j}=\tilde \beta_{i,j}\max(1,\beta_{i,j})$).  
Therefore
\begin{eqnarray}\label{thres4}
Z_{\La_\g^c}^{\g,\beta}(t_{\g})\le Z_{\La_\g^c}^{\g,\tilde\beta}(t_{\g}) \prod_{e\sim \gamma} \sqrt{\max(\beta_{e},1)}
\end{eqnarray}
with $Z_{\La_\g^c}^{\g,\tilde\beta}(t_\g)$ defined as in (\ref{zzz}) with $\epsilon, \beta$ replaced by $\te,\tilde\beta$.
Hence we can replace 
$\beta$ by $\tilde \beta$ at the cost of the term $\prod_{e\sim \gamma} \sqrt{\max(\beta_{e},1)}$.

%Now, $Z_{\La_\g^c}^{\g,\tilde\beta}(t_\g)$ approximates the normalization constant
%$$
%Z_{\La_\g^c}^{\te}=
%1=\nu_{\La_\g^c}^{\te,\tilde\beta}\left(\sqrt{\mathsf{det}\, D_{\La_\g^c}^{\te,\tilde\beta}}\right),
%$$
%with the difference that, in the former,  the measure considered  is $\nu^{\delta_0,\tilde\beta}$ instead of $\nu^{\te,\tilde\beta}$, and there is a multiplicative factor $e^{-F_{\partial\g}^{\tilde\beta}(\nabla t)}$ (which is helpful, since we aim to upper bound $Z_{\La_\g^c}^{\g,\tilde\beta}(t_\g)$). 
\noindent The following lemma, which adapts Lemma 3 \cite{ds}, provides an upper bound of $Z_{\La_\g^c}^{\g,\tilde\beta}(t_\g)$. 
%It uses in a crucial way the fact that the measure $\mu_{\Lambda_\gamma^c}^{\tilde\epsilon,\tilde\beta}$ is a probability
%measure.
\begin{lemma}
\label{estz}
For any configuration of $t_\gamma=t_{| \Lambda_\gamma}$, 
$Z_{\La_\g^c}^{\g,\tilde\beta}(t_\g)\le e^{\sum_{e\sim \gamma}\tilde\be_{e}}.$
\end{lemma}
\begin{proof}
We have
$$
Z_{\La_\g^c}^{\g,\tilde\beta}(t_\g)
=\int \left({\prod_{j\in \Lambda_\g^c} dt_j\over \sqrt{2\pi}}\right) e^{-F_{\Lambda_\g^c}^{\tilde\be}(\nabla t)-F^{\tilde\be}_{\partial \gamma}(\nabla t)}
\sqrt{ \det( D_{\Lambda_\g^c}^{\tilde\e, \tilde\beta})}
$$
Let $t^*=\max\{t_k, \; k\in \Lambda_\g\}$. We perform the following translation in the variables $t_j\rightarrow t_j+t^*$ for $j\in \Lambda_\g^c$:
in the previous integral the term $F_{\Lambda_\g^c}^{\tilde\beta}(\nabla t)$ does not change, the term $F^{\tilde\beta}_{\partial \gamma}(\nabla t)$ becomes
$$
\sum_{k\in\La_\g, j \in\La_\g^c, k\sim j}\tilde\be_{kj}(\cosh(t_j+t^*-t_k)-1),
$$
and the term $\det(D_{\Lambda_\g^c}^{\tilde\e, \tilde\beta})$ is replaced by $\det(D_{\Lambda_\g^c}^{e^{-t^*}\tilde\e, \tilde\beta})$.
Since $t^*-t_k\ge0$ we have
$$
\cosh (t_j+t^*-t_k)-1\ge e^{t_k-t^*}(\cosh(t_j)-1)+(e^{t_k-t^*}-1).
$$
This implies that
$$
\sum_{k\in\La_\g, j \in\La_\g^c, k\sim j}\tilde\be_{kj}(\cosh(t_j+t^*-t_k)-1)\ge M_{\Lambda_\g^c}^{e^{-t^*} \tilde \e}(t) + 
\sum_{k\in\La_\g, j \in\La_\g^c, k\sim j} \tilde\be_{k,j} (e^{t_k-t^*}-1),
$$
and
$$
Z_{\La_\g^c}^{\g,\tilde\beta}(t_\g)\le
e^{\sum_{k\in\La_\g, j \in\La_\g^c, k\sim j} \tilde\be_{k,j} (1-e^{t_k-t^*})}  \mu_{\Lambda_\g^c}^{e^{-t^*}\tilde\e,\tilde\be}(1) 
\le e^{\sum_{e\sim \gamma}\tilde\be_{e}},
$$
using that $\mu_{\Lambda_{\g}^c}^{e^{-t^*}\tilde\e,\tilde\be}$ is a probability measure.
\end{proof}
Combining \eqref{estmu}, (\ref{thres4}), Lemma \ref{estz}, and integration on the variables $(\nabla t_e)_{e\in \gamma}$, we obtain
that
$$
\mu_{\La}^{\eta\delta_0, \beta}\left( e^{t_x/2}\right) \le 
I_{\eta}\sum_{\g\in\G_x}\left( \prod_{e\sim \gamma} \sqrt{\max(\beta_{e},1)}e^{\min(\beta_e,1)}\right)
\left(\prod_{e\in \gamma} I_{\beta_e}\right)
$$
We set 
$$
\hat I_a =\E(I_\beta)={\Gamma(a+\demi)\over \Gamma(a)} \int_{-\infty}^\infty \cosh(t)^{-a-\demi} dt,
$$
and
$$
\hat J_a = \E(\max (\beta,1)e^{\min(\beta,1)}).
$$
where $\beta$ is a gamma random variable with parameter $a$. Clearly, $\hat I_a$ and $\hat J_a$ tend respectively
to 0 and 1 when $a$ tends to 0. 
Integrating on the random variables $\beta_e$ we deduce that there exists a constant $C_0$ such that 
$$
\E\left(\mu_\La^{\eta\delta_0,\be}(e^{t_x/2})\right)\le C_0 \left( \hat I_a (\hat J_a)^{2d-2}(2d-1)\right)^{\vert x\vert}, 
$$
$$
\E\left((\beta_{x,y})^{{1\over 4}}\mu_\La^{\eta\delta_0,\be}(e^{t_x/2})\right)\le C_0 \left( \hat I_a (\hat J_a)^{2d-2}(2d-1)\right)^{\vert x\vert}, 
$$
for any $y\sim x$.
This provides the exponential decrease for all $a>0$ such that 
$$
\hat I_a (\hat J_a)^{2d-2}(2d-1)<1.
$$
\end{proof}
\begin{proof}(Corollary \ref{rec2})
For any connected finite set $\Lambda$ containing 0, by Theorems \ref{annealed} and \ref{meas}, 
the ERRW on $\Lambda$ starting 0 and with constant initial parameter $a$
is a mixture of reversible Markov chains in conductance
$c_{x,y}=\beta_{x,y} e^{t_x+t_y}$, where $\beta_{x,y}$ are Gamma$(a,1)$ independent random variables. 
As in Corollary \ref{rec}, there exists a constant $C>0$ such that
\begin{eqnarray*}
\E\left(\mu_\La^{\delta_0,\beta}( (c_{x,y}/c_{\delta,0})^{1/4})\right)
&\le& 
C\E\left( (\beta_{x,y})^{{1\over 4}}\left[\mu_\La^{\delta_0,\beta}(e^{t_x/2})\mu_\La^{\delta_0/2,\beta}(e^{t_y/2})\right]^{1/2} \right)
\\
&\le&
C\E\left( (\beta_{x,y})^{{1\over 4}}\mu_\La^{\delta_0,\beta}(e^{t_x/2})\right)^{1/2} 
\E\left((\beta_{x,y})^{{1\over 4}}\mu_\La^{\delta_0/2,\beta}(e^{t_y/2})\right)^{1/2}
\end{eqnarray*}
The rest of the proof is similar to the proof of Corollary \ref{rec}.
\end{proof}

\begin{proof} (Corollary \ref{trans})
Fix $d\ge3$. Let $\La_n=\{i\in \Z^d, \; \| i\|_\infty \le n\}$ be the ball centred at $0$ with radius $n$ and 
$\partial \Lambda_n=\{i\in \Z^d, \; \| i\|_\infty = n\}$ its boundary. Let $E_n$ be the set of edges contained in $\Lambda_n$.
Disertori, Spencer and Zirnbauer prove in Theorem 1 \cite{dsz} (see also their remark above) that, for any $m>0$, there exists $\tilde{\be}_c(m)$ such that, for any $\beta>\tilde \beta_c(m)$, for any $n\in\N$, $x$, $y$ $\in\La_n$, 
\begin{equation}
\label{dsz}
\mu_{\La_n}^{\de_0,\be}\left(\cosh^m(t_x-t_y)\right)\le2;
\end{equation}
the result is stated for constant pinning, but its proof does not require that assumption, as we checked through careful reading. 

We consider the VRJP on $\Z^d$ with constant conductances $W_{i,j}=\beta$ and denote by
  $\P_0^\beta(.)$ its law starting from $0$. We denote by $P_0^c$ the law of the Markov chain in conductances
$c_{i,j}= \beta e^{t_i+t_j}$ starting from 0, where $(t_i)$ is distributed according to $\mu_{\Lambda_n}^{\delta_0, \beta}$.
Let $H_{\partial\Lambda_n}$ be the first hitting time of the boundary $\partial \Lambda_n$ and $\tilde H_0$ be the first
return time to the point $\delta$.
Let $R(0,\partial\La_n)$ (resp.  $R(0,\partial\La_n,c)$) be the effective resistance between $0$ and $\partial\La_n$ for conductances $1$ (resp. $c_{i,j}$). Classically
$$
 c_{0} R(0, \partial \Lambda_n,c) = {1\over P^c_0 (H_{\partial \Lambda_n}<\tilde H_0)}
$$
with $c_0=\sum_{j\sim 0} c_{0,j}$.
By Theorem \ref{meas} and Jensen's inequality,
\begin{eqnarray}
 {1\over \E^\beta_0( H_{\partial\Lambda_n} <\tilde H_0 )}\le
\mu_{\Lambda_n}^{\delta_0, \beta} ({1\over P^c_0 (H_{\partial \Lambda_n}<\tilde H_0)}) 
&\le& 
 \mu_{\Lambda_n}^{\delta_0, \beta} ( c_0 R(0, \partial \Lambda_n,c))
 \label{inbd}
% // &=& 1+ \mu_{\Lambda_n}^{\delta_0, \beta} ( e^{t_0} R(0, \partial \Lambda_n,c))
\end{eqnarray}
Let us now show that for all $\be>\tilde{\be}_c(2)$,
\begin{eqnarray}\label{Resistance}
\mu_{\La_n}^{\de_0,\be}\left[c_0 R(0,\partial\La_n,c)\right]\le 16d R(0,\partial\La_n)
\end{eqnarray}
This will enable us to conclude: since $\limsup R(0,\partial \Lambda_n) <\infty$, \eqref{inbd} and \eqref{Resistance} imply that $ \Pb^\beta_0( \tilde H_0 =\infty)>0$.

Indeed, let $\theta$ be the unit flow from 0 to $\partial \Lambda_n$ which minimizes the $L^2$ norm. Then
$$
R(0,\partial \Lambda_n, c) \le \sum_{\{i,j\}\in E_n} {1\over c_{i,j}}  \theta^2(i,j),
$$
and 
$$
R(0,\partial \Lambda_n) = \sum_{\{i,j\}\in E_n} \theta^2(i,j).
$$
Now, for all $\beta>\tilde \beta_c(2)$, using identity \eqref{dsz},
$$
 \mu_{\Lambda_n}^{\delta_0, \beta} ( {c_0\over c_{i,j}})\le \sum_{l\sim0} \mu_{\Lambda_n}^{\delta_0, \beta}(e^{2(t_0-t_i)})^{1/2}
  \mu_{\Lambda_n}^{\delta_0, \beta}(e^{2(t_l-t_j)})^{1/2} 
 \le
 16d
 $$
 %where we choose $C$ such that $ 3u\le \demi (\cosh(u)-1) + 3\log C$, which implies 
 %$$ \mu_{\Lambda_n}^{\delta_0, \beta}(e^{-3t_0})\le
 %C^3\mu_{\Lambda_n}^{\delta_0, \beta}(e^{\demi(\cosh(t_0)-1)})\le 2C^3.
 %$$
 
\end{proof}
\medskip
\noindent
{\bf Acknowledgment.} We are particularly grateful to Krzysztof Gawedzki for a helpful discussion on the hyperbolic sigma model, and for pointing out reference \cite{dsz}. We would also like to thank Denis Perrot and Thomas Strobl for suggesting a possible link between the limit measure of VRJP and sigma models. We are also grateful to Margherita Disertori for a useful discussion on localization results on the hyperbolic sigma model.

\footnotesize
\bibliographystyle{plain}
\bibliography{errw-cvrrw}

\end{document}